\documentclass{elsarticle}
\usepackage[latin9]{inputenc}
\usepackage{a4}
\usepackage{amsfonts}
\usepackage{amssymb}
\usepackage{amsmath}
\usepackage{amsthm}
\usepackage{changepage}
\usepackage{nomencl}
\usepackage{geometry}
\usepackage{breqn}
\usepackage[toc,page]{appendix}
\begin{document}
\providecommand{\keywords}[1]{\textbf{\textit{Keywords: }} #1}
\newtheorem{satz}{Theorem}
\newtheorem{lemma}[satz]{Lemma}
\newtheorem{prop}[satz]{Proposition}
\newtheorem{kor}[satz]{Corollary}
\theoremstyle{definition}
\newtheorem{defi}{Definition}

\newcommand{\cc}{{\mathbb{C}}}   
\newcommand{\ff}{{\mathbb{F}}}  
\newcommand{\nn}{{\mathbb{N}}}   
\newcommand{\qq}{{\mathbb{Q}}}  
\newcommand{\rr}{{\mathbb{R}}}   
\newcommand{\zz}{{\mathbb{Z}}}  
\author{Joachim K\"onig} 
\title{On rational functions with monodromy group $M_{11}$} 
\address{Universit\"at W\"urzburg, Emil-Fischer-Str.\ 30, 97074 W\"urzburg, Germany}
\ead{joachim.koenig@mathematik.uni-wuerzburg.de}
\date{}
{\abstract{We compute new polynomials with Galois group $M_{11}$ over $\mathbb{Q}(t)$. These polynomials stem from various families of covers of $\mathbb{P}^1\mathbb{C}$ ramified over at least 
4 points. Each of these families has features that make a detailed study interesting. Some of the polynomials lead, via specialization, to number fields with very small discriminant or few
ramified primes.}}
\maketitle
\section{Theoretical background}
As the following sections make use of Hurwitz spaces as moduli spaces for families of covers of the projective line, I will begin with a very brief outline of the theory. 
More thorough introductions may be found e.g.\ in \cite{DF2} or \cite{Voe}.\\\\
Let $G$ be a finite group and $C=(C_1,...,C_r)$ be an $r$-tuple of conjugacy classes of $G$, all $\ne \{1\}$. The Nielsen class of the class tuple $C$ is defined as
$$Ni(C) = \{(\sigma_1,...,\sigma_r) \in G^r\mid \langle \sigma_1,...,\sigma_r\rangle = G, \sigma_1\cdots \sigma_r=1, \exists_{\pi \in S_r}: \ \sigma_i \in C_{\pi(i)} (i=1,...,r)\}.$$
The inner Nielsen class $Ni^{in}(C)$ is defined as the quotient of $Ni(C)$ by the diagonal action of $Inn(G)$.
The analogous sets with the additional requirement $\pi=id$ are called the straight Nielsen classes $SNi(C)$ resp.\ $SNi^{in}(C)$.\\
%
It is well known that the elements of $Ni^{in}(C)$ parametrize branched Galois covers of $\mathbb{P}^1\cc$ in the following way:\\
Let $\{p_1,...,p_r\}$ be a fixed subset of $\mathbb{P}^1\cc$ of cardinality $r$, $p_0\in \mathbb{P}^1\cc\setminus\{p_1,...,p_r\}$ and 
$f:\pi_1(\mathbb{P}^1\cc\setminus\{p_1,...,p_r\},p_0)\to G$ an epimorphism (which, by Riemann's existence theorem, induces a Galois cover) 
mapping the $r$-tuple of standard fundamental group generators to an element of $Ni(C)$.\\ Now $f$ and $f'$ (with same branch point sets and same base point) are defined to by equivalent
if for some $\gamma \in \pi_1(\mathbb{P}^1\cc\setminus\{p_1,...,p_r\},p_0)$: $f(\gamma \delta \gamma^{-1}) = f'(\delta)$ for all 
$\delta \in \pi_1(\mathbb{P}^1\cc\setminus\{p_1,...,p_r\},p_0)$.\\
Letting the branch points vary over all $r$-sets in $\mathbb{P}^1\cc$, the set of these equivalence 
classes forms a (not necessarily connected) algebraic variety (for non-empty $Ni(C)$, of course), known as the inner Hurwitz space $\mathcal{H}^{in}(C)$. This variety comes with a natural morphism
$\Psi: \mathcal{H}^{in}(C) \to \mathcal{U}_r$ to the space $\mathcal{U}_r$ of $r$-subsets of $\mathbb{P}^1\cc$, mapping (an equivalence class of) a cover to the set of its branch points.\\
The elements of a given fiber of $\Psi$ correspond one-to-one to the elements of $Ni^{in}(C)$.\\
If $Z(G)=\{1\}$, then by a famous theorem of Fried and V\"olklein (\cite[Corollary 1]{FV}), $G$ occurs as the Galois group of a regular Galois extension of $\qq(t)$ iff the inner 
Hurwitz space $\mathcal{H}^{in}(C)$ contains a rational point for some class tuple $C$ of $G$. In other words, inner Hurwitz spaces are fine moduli spaces under the assumption $Z(G)=\{1\}$.\\ \\
{\bf Remark}:\\ We restrict here to the notion of inner Hurwitz spaces. In certain contexts, {\it absolute} Hurwitz spaces are more natural. However, for the purpose of this paper,
there is no difference anyway, since the Mathieu group $M_{11}$ has no outer automorphisms. 
\\ \\
To find out about the existence of rational points, one needs to investigate the algebraic structure of the Hurwitz spaces. 
The dimension of the Hurwitz spaces can be reduced by 3 via the action of $PGL_2(\cc)$ on $\mathbb{P}^1\cc$, which induces an equivalence relation on $\mathcal{U}_r$, and thereby also on $\mathcal{H}^{in}(C)$.
Under relatively mild assumptions, rational points on the reduced Hurwitz
spaces also lift to rational points on the non-reduced ones. Especially for $r=4$, reduced Hurwitz spaces are curves.
The genera of these curves, and more generally, of certain curves on Hurwitz spaces of higher dimension,
can be computed from the action of the Hurwitz braid group $\mathcal{H}_r$ on $Ni^{in}(C)$, 
which acts on $Ni(C)$ via $(\sigma_1,...,\sigma_r)^{B_i} = (\sigma_1,...,\sigma_{i-1},\sigma_i\sigma_{i+1}\sigma_i^{-1},\sigma_i,...,\sigma_r)$, $i=1,...,r-1$, where $B_1,...,B_{r-1}$
are the standard generators of the braid group, as defined e.g.\ in \cite[Chapter III.1.2]{MM}.\\
This action has an interpretation as a monodromy action on the fibers of the branch point reference cover $\Psi: \mathcal{H}^{in}(C) \to \mathcal{U}_r$.
\\ Similarly, for different versions of covers of reduced Hurwitz spaces over suitable parameter spaces,\footnote{If some of the classes $C_1$,...,$C_r$ occur several times in the class tuple $C$,
there are various ways to ``symmetrize" the branch points, leading to different variants of reduced Hurwitz spaces. I will not elaborate on this here; 
for a detailed outline cf.\ e.g.\ \cite[Chapter III.7]{MM}.} the appropriate choice of braids yields the monodromy action of the cover.
Therefore, the cycle types of the monodromy group generators 
(which depend on the exact version of the parameter space, e.g.\ on symmetrization of branch points) yield explicit genus formulas for the Hurwitz curves. 
Cf.\ III. (5.11) and Theorem III.7.8 in \cite{MM} for such formulas.\\ 
In particular, for {\it rational} class 4-tuples $C$, if $\mathcal{H}^{in}(C)$ is connected (or more generally, if there is a rigid orbit in the braid group action) and the braid genus is zero, 
with some oddness condition satisfied (assuring that the genus zero curve is actually a rational curve), then the existence of infinitely many rational points on $\mathcal{H}^{in}(C)$ follows.\\
For larger genus (or for varieties of dimension $>1$), explicit computation may help to clarify the situation. Hopefully, this paper supports that there is some value in explicit computation of Hurwitz families.
\section{Overview of known $M_{11}$ extensions of $\qq(t)$}
\label{overview}
In the following, we consider Galois extensions with Galois group the Mathieu group $M_{11}$. 
This group is the smallest sporadic simple group (of order $7920=11\cdot 10\cdot 9\cdot 8$) and has a sharply 4-transitive action on 11 points.\\
There are several known ways to obtain $M_{11}$ as a (regular) Galois group over $\qq(t)$. The first known example (cf.\ \cite[Chapter I.9.4]{MM}) used a 3-point cover defined over $\qq$ 
with regular Galois group $M_{12}$ and genus zero monodromy to obtain the point stabilizer $M_{11}$ via suitable specialization. The same idea, starting with a family of 4-point 
covers with Galois group $M_{12}$ was used by Malle in \cite{Malle} (Section 10).\\ \\
It is also known that $M_{11}$ itself has a family of genus zero (in the action on 12 points!) covers defined over $\qq$ with 4 ramification points (cf.\ Theorem 3.10 in \cite{De}).\\
In other words, there is a family of rational functions $t_a=\frac{f_a(x)}{g_a(x)}\in \qq(a)(t)[x]$ with monodromy group $M_{11}$ (by this, we mean that 
$Gal(f_a(x)-tg_a(x)\mid \qq(a,t)) \cong M_{11}$), where $t$ and $a$ are independent transcendentals. The reason for this is the existence of a rational ($S_3$-symmetrized) Hurwitz curve
for the Nielsen class of $M_{11}$-generating class tuples of type $(2A,3A,3A,3A)$.
\\ \\ In the following, we focus on families of $M_{11}$-covers 
with {\it non-rational} Hurwitz curves and show the existence of rational points on them.
\section{Rational functions with $M_{11}$-monodromy and more than 3 branch points}
The families of covers considered in the following are families of genus zero (with regard to the suitable permutation action); this means that each member can be defined
(a priori over $\cc$) by a polynomial equation of the form $f(x)-t\cdot g(x)=0$ (with $f,g\in \cc[x]$). More precisely, the entire family is a (non-Galois!) genus zero cover of 
$\mathcal{H}(C)\times \mathbb{P}^1$, with $\mathcal{H}(C)$ the corresponding Hurwitz space, and can be defined by a polynomial equation $F(x)-t\cdot G(x)=0$, 
where the coefficients of $F$ and $G$ are algebraic functions in the function field of the Hurwitz space. With the correct choice of conjugacy classes (e.g.\ all classes rational, 
which will be the case for all the examples) this function field is defined over $\qq$.
\\ \\
This yields a basic plan for explicit computations. First, find a single polynomial with the correct monodromy - either modulo a small prime via exhaustive search (an approach that is
convenient if the degrees and therefore the number of variables involved are relatively small), or as a
complex approximation, obtained from an initial 3-point cover via deformation techniques. Next, gain sufficiently many covers (via Newton iteration in $\cc$ or $p$-adic lifting)
to interpolate and find algebraic dependencies between suitable coefficients (or combinations of coefficients) of $F$ and $G$. As the degree of these dependencies is not fixed, one 
should first specialize appropriately to find out which coefficients lead to which degrees. Finally, the coefficients of these dependencies can be recognized as algebraic (and hopefully
rational) numbers via the LLL algorithm. Once a precise polynomial equation over $\qq$ is known, one can search for rational solutions and then check whether these solutions are 
``good" in the sense that they correspond to non-degenerate covers in our family.
These techniques will be illustrated with the concrete examples below, see especially Section \ref{2235}. 
\\ \\
Detailed outlines and computational examples of such techniques can be found in \cite{Cou}, \cite{Koe} (cf.\ especially Chapters 5-8 for explicit computations) or \cite{Malle}.
Notable applications of these techniques include the explicit computation of $M_{24}$-polynomials by Granboulan (\cite{G}) and M\"uller (\cite{Mue}).
\subsection{A family with four branch points and elliptic Hurwitz curve}
\label{2235}
In the following, consider $M_{11}$ in its primitive permutation action on 12 points. Let $2A$ be the conjugacy class of elements of cycle structure $2^4.1^4$ in $M_{11}$, $3A$ be the class of elements of cycle structure
$3^3.1^3$ and $5A$ be the class of cycle structure $5^2.1^2$. The Riemann-Hurwitz genus formula shows that the class tuple $(2A,2A,3A,5A)$ is a genus zero tuple in this degree 12 
action.\\ \\
Let $SNi^{in}(2A,2A,3A,5A):=\{(\sigma_1,...,\sigma_4)\in M_{11}^4 \mid \sigma_1,\sigma_2 \in 2A, \sigma_3\in 3A, \sigma_4\in 5A, 
\langle\sigma_1,...,\sigma_4\rangle=M_{11}, \sigma_1\cdots \sigma_4=1\}/Inn(M_{11})$ the straight inner Nielsen class of $M_{11}$-generating 4-tuples of the prescribed cycle type.\\
This set forms a single orbit (of length 100) under the action of the braid group $\langle B_1, B_2^2, B_3^2\rangle$ 
(the stabilizer in the Hurwitz braid group $\mathcal{H}_4$ of this ordering on the conjugacy classes). The usual braid genus criteria  (cf.\ \cite[Thm. III.7.8a)]{MM}) 
yield that the 
$C_2$-symmetrized Hurwitz curve $\mathcal{C}$ is of genus 1. In order to find out whether it contains rational points, I explicitly computed a model for the corresponding family 
of $M_{11}$-covers, using deformation techniques.\\

As a starting point, one finds a polynomial with the required ramification structure and Galois group $M_{11}$ over the field $\mathbb{F}_7(t)$ (the prime 7 is chosen because it is the smallest
prime not dividing $|M_{11}|$). This can simply be done by exhaustive search, running through all polynomials $f(x,t):=(x^2+a_1x+a_2)^5(x^2+a_3x+a_4)-t(x^2+a_5)^3(x^3+x^2+a_6x+a7) \in \mathbb{F}_7(t)[x]$.
A solution is given by $(a_1,...,a_7):=(0,1,0,2,3,5,1)$. The corresponding field extension of $\mathbb{F}_7(t)$ is ramified over $t\mapsto 0$, $-1$, $-2$ and $\infty$.\\
Next, lift this solution to many different $7$-adic solutions, ramified over $0$, $-1+k\cdot 7$, $-2-k\cdot 7$ and $\infty$ (for various $k\in \zz$). 
Hereby assume, as indicated in the above polynomial $f(x,t)$, that the place $x\to\infty$ lies over $t\to \infty$ with ramification index 3. This is possible without loss, because one of the
3-cycles of an element $\sigma$ of order 3 in $M_{11}$ is fixed by $N_{M_{11}}(\sigma)$, which means that over any field of definition $K$, the corresponding place in $K(x)|K(t)$
is fixed by the decomposition group and therefore a rational place. Furthermore, the second coefficients of two of the factors of $f$ can be fixed to $0$ and $1$, respectively (as above), via
linear transformations.\\
Under these restrictions, there will be a unique $7$-adic solution for each $k$. This lifts e.g.\ the above coefficients $a_i \in \mathbb{F}_7$ to many $a_{i,k} \in \zz_7$. Now all the $a_{i,k}$, $k\in \zz$, are specializations of a coefficient
$A_i\in \overline{\qq}(\mathcal{C})$, with $\overline{\qq}(\mathcal{C})$ the function field of the Hurwitz curve (which is actually defined over $\qq$ in our case). We can therefore obtain an algebraic dependency between any two
of the $A_i$ via interpolation. The degree of these dependencies is a priori unknown. In our case, $A_1$ and $A_3$ turned out to fulfill a polynomial equation of degrees 15 and 13, respectively.\\
This equation is given in the appendix. Note that a priori, the coefficients of the equation are of course over $\mathbb{Q}_7$ (expanded to a given precision). It is however easy to retrieve the
rational numbers from sufficiently precise $p$-adic expansions.\\
Since we know that our function field has genus 1, there must be much simpler defining equations. These can be found via computation of 
Riemann-Roch spaces; as our curve has rational points, this simply amounts to bringing an elliptic curve into Weierstrass normal form.\footnote{In fact, as the equation in $A_1$ and $A_3$ has
rather large degrees and coefficients, it turned out to be difficult to directly compute Riemann-Roch spaces with Magma. Therefore, these calculations were in fact first done modulo $p$, for
several primes $p$. The rational equations were then retrieved via the Chinese remainder theorem.}\\
In our case, one obtains the Weierstrass form $W^2=Y^3-675Y-1250$ for the elliptic curve $\mathcal{C}$.
Explicit algebraic dependencies which enable one to express the functions $W,Y\in \qq(\mathcal{C})$ as rational functions in the coefficients $A_1$ and $A_3$ are given in the appendix.
Now, one easily expresses all the coefficients of our model as rational functions in $W$ and $Y$, by interpolation as above. This yields an equation for the universal family in the form
$F(x)-t\cdot G(x)=0$, with $F,G\in \qq(W,Y)[X]$, which is however too lengthy to conveniently fit into this paper.
\\
Finally, we can specialize to suitable rational points $(w_0,y_0)$ in the above model of $\mathcal{C}$, to obtain an explicit $M_{11}$-polynomial over $\qq(t)$.
One such point (namely, $(w_0,y_0) = (-300,50)$) yielded the following, unexpectedly nice, polynomial:
\begin{satz}
\label{m11small}
Let $f(x,t):= (x^2-x-1)^5(x^2-x-1/16)-t\cdot x^3(x-1)^3(x^3-2) \in \qq(t)[x]$. Then $f$ has Galois group $M_{11}$. 
The branch cycle structure is of type $(2A,2A,3A,5A)$. 
\end{satz}
\begin{proof}
To show $Gal(f) \ge M_{11}$, simply specializing $t$ and reducing modulo suitable primes (to get sufficiently many 
cycle types in the Galois group by Dedekind's criterion) suffices. Now one needs to exclude the proper overgroups 
of $M_{11}$ in $S_{12}$, i.e. $M_{12}$, $A_{12}$ and $S_{12}$. What sets $M_{11}$ apart from all of these is the 
existence of a subgroup of index 11 (the stabilizer $M_{10}$ in the natural degree 11 permutation representation).
This stabilizer has an index 2 subgroup which acts intransitively on 12 points, with two orbits of length 6.
As the specialization $t\mapsto 81/16$ leads to a reducible polynomial with factors of degree 6,
we can develop the roots of $f$ in the Laurent series field $\mathbb{F}_p((t-81/16))$, where 
$p$ is any prime such that the degree 6 polynomials split completely. Then let $a$ be the sum of 6 
such roots (of course belonging to the same degree 6 factor in the non-reduced polynomial); we expect $a$ to have 
a minimal polynomial of degree 22 over $\mathbb{F}_p(t)$, and equating coefficients of the Laurent expansions yields this 
polynomial. Doing the same for suitably many primes, the Chinese remainder theorem also yields the analagous polynomial
(with rational coefficients) for the non-reduced sum of 6 roots.\\
Now one simply verifies that the original polynomial $f$ factors over the function field of the degree 22 polynomial;
this proves $Gal(f)\cong M_{11}$.\\ \\
The assertion about the branch cycle structure can be easily verified.
\end{proof}
The above proof method also yields an explicit degree 11 polynomial with Galois group $M_{11}$ and 
ramification type $(2A,2A,3A,5A)$. As this is a genus-1 tuple on 11 points, one cannot hope to obtain a polynomial 
$g(x,t)$ linear in $t$; however, upon suitable transformations one obtains a polynomial of degree 2 in $t$. One such polynomial is given in the following lemma.
Furthermore, by specializing $t$ to many different rational values, applying the Magma function \texttt{OptimizedRepresentation} and then interpolating between
all ``similar-looking" optimized polynomials, one also obtains a degree-11 polynomial with very small coefficients (and of degree 3 in $t$):
\begin{lemma}
Let $g(x,t):=1/20(x+1)(x-179)(x^3 - 12x^2 + 648x - 464)^3
-(x+1)(7x^7 - 132x^6 + 6912x^5 - 74352x^4 + 822272x^3 - 1104000x^2 - 22464000x - 24883200)t
-x^5(x-8)t^2$ and \\
$h(x,t):=x^{11}+x^7(3x+2)t+x(3x^4 + 14/5x^3 + 4/5x^2 - 40/81x - 16/81)t^2-(x-2/5)(x+2/5)t^3\in \qq(t)[x]$. 
Then $g$ and $h$ have Galois group $M_{11}$ in its natural degree 11 action on the roots.
\end{lemma}
{\bf Remarks}:
\begin{itemize}
\item There seemed to be no reason why a member of this family would have such nice and small coefficients as the polynomial $f$ in Theorem \ref{m11small}. 
It would therefore be interesting to know whether there is a ``natural" explanation for the existence of this polynomial (possibly one that could generalize to other Mathieu groups).
\item 
The elliptic curve given by $W^2=Y^3-675Y-1250$ is of rank 1, so there are infinitely many $\qq$-points; and as only finitely many of those
 can correspond to degenerate covers, there will be infinitely many equivalence classes of $M_{11}$-polynomials over $\qq$ in this family! 
 Unfortunately, the rational point that leads to the above polynomial $f$ seems to be the only one where the coefficients in our computational model have small height. 
 Nevertheless, explicit algebraic dependencies between
 any of those coefficients and the Weierstrass $\wp$ function of the elliptic curve make it possible to find many more rational points on the curve and therefore $M_{11}$-polynomials.\\
 We conclude this section with a polynomial from this family which possesses specializations with totally real Galois closure. The large rational numbers that occur as coefficients 
 could hardly have been found via exhaustive search, but are easy to find once a dependency betweeen coefficients of our model and the Weierstrass $\wp$ function is known.
 \end{itemize}
\begin{satz}
Let $h(t,x):=h_1(x)^5\cdot h_2(x)- t\cdot h_3(x)^3\cdot h_4(x)$, where
\begin{center}
$h_1(x):=217019x^2 + 87907368140554183100162x + 4508281300847688731169319769769841835444$,\\
$h_2(x):=37258218541488534910719399x^2 + 246529569778532375140876985021306849541294312x + 14635567009082568185596461891292367332677180975427256879233264$,\\
$h_3(x):=2223585235421468919x^2 + 274156098050901691661379937349434622x + 8443409213944751989643951279235627283533452107433272$,\\
$h_4(x):=13385859493702x^3 + 225702085730490664288660568855789598873301353168x + 16436156522418265669179020769359057896808189915554213257292620245$.
\end{center}
Then $h$ has Galois group $M_{11}$ over $\qq(t)$, and for all $t_0\in (2.65\cdot 10^{39}, 3.42\cdot 10^{39})$, the specialized polynomial $h(t_0,x)$ has only real roots.
\end{satz}
\begin{proof}
We only show the assertion about totally real number fields. This is, however, easy, as it suffices to verify it with the computer for one specialization $t_0$ in the above interval.\\
The bounds of this interval are (approximately) the non-zero finite real branch points of $h$. As the number of real roots can only change at a branch point, the assertion follows. 
\end{proof}

\subsection{A family with five branch points}
Here we consider the class 5-tuple $(2A,2A,2A,2A,3A)$ in $M_{11}$. In the action on 12 points, this is a genus zero tuple. The Nielsen class $SNi^{in}(2A,2A,2A,2A,3A)$ of 
$M_{11}$-generating tuples is of length 2376 (cf.\ Table 2 in \cite{MSV}) and forms a single orbit under the Hurwitz braid group. By fractional linear transformations, we can 
(at least generically) fix the branch points of the corresponding $M_{11}$-covers to $\infty$ (for the element of order 3) and the roots of a degree 4 polynomial of the form 
$x^4+ax^2+ax+b$, with $a,b\in \cc$. The corresponding reduced Hurwitz space is then a surface defined over $\qq$. Even though the standard braid genus criteria seem to yield no
obvious rational curves on this surface, we can still hope to find points by explicit computation.\\
As in the previous section, lifting an initial approximate solution, i.e.\ a solution modulo a prime $p$  (here for $p=7$) to many different $p$-adic solutions with different branch point locus yielded, 
via interpolation, an algebraic dependency - this time between three suitable coefficients in the model (as the Hurwitz variety is two-dimensional), of degrees 14, 16 and 19.\\ Searching for points on this variety yields several ``bad" solutions (e.g.\ corresponding to degenerate
covers with fewer than 5 branch points and smaller Galois group); but there are also ``good" rational points, and with rather small coefficients. Two of those are given in the following theorem.
It would be interesting to know if there are infinitely many non-equivalent covers defined over $\qq$ in this family, and if the Hurwitz space is maybe in fact a rational surface.
\begin{satz}
\label{m11_5er}
The polynomials
$f(x,t):=(x^{12} + 24x^{11} + 204x^{10} + 912x^9 + 2676x^8 + 4032x^7 + 9056x^6 - 1920x^5 + 7728x^4 - 16512x^3 - 20544x^2 - 6912x - 1088)
-t(x^2-2)^3(x^3+3x^2+6x+2)$ and 
$g(x,t):=(x^4 + 4/3x^3 + 2/3x^2 - 1/11)^2(x^4 + 22/3x^3 - 8/3x^2 - 2x + 13/11)-t(x^2-3/11)^3(x^3+x^2+1/11x-1/11)$
have Galois group $M_{11}$ over $\qq(t)$. In both cases, the branch cycle structure is of type $(2A,2A,2A,2A,3A)$.
\end{satz}
\begin{proof}
For both polynomials, the proof can be carried out in analogy with Theorem \ref{m11small}.
\end{proof}
Note that the polynomials $f$ and $g$ in the previous theorem are not equivalent (via fractional linear transformations); in particular, the finite ramification points of $f$ are
the roots of an irreducible degree 4 polynomial, while $g$ has a rational branch point at $t=0$.
\subsection{Rational functions of degree 11}
The group $M_{11}$ is the monodromy group of rational functions over $\qq$ of degree 11 as well as 12. For degree 12, this has been seen in several different ways already.\\
For degree 11, the tuple of classes with cycle structures $(2^4.1^3,2^4.1^3,3^3.1^2,4^2.1^3)$ 
has a Hurwitz curve of genus 2. Explicit computations of this curve yielded a ``good" rational point, allowing a non-degenerate cover of this family defined over $\qq$: 
\begin{satz}
\label{2234}
The polynomial
$f(t,x):=(77x^3 + 10989x^2 + 129816x + 496368)^3(77x^2 + 2376x + 15472)-t(11x^2-1296)^4(11x^2+143x+621) \in \qq(t)[x]$ has Galois group $M_{11}$ over $\qq(t)$.
\end{satz}
In fact one can show that up to equivalence (i.e.\ fractional linear transformations in $t$ and in $x$), this is the {\it only} polynomial in this family which is defined over $\qq$.\\
The reason is that the genus 2 Hurwitz curve turns out to be birationally equivalent to the hyperelliptic curve given by $y^2=(x^2 - x + 3)(x^2 + 1)(x^2 + x + 1)$. For such curves,
there are methods to explicitly determine, under a few assumptions, the complete set of rational points. Using Magma, we found that the Jacobian of the above curve is of rank 1,
and Chabauty's method (as described e.g.\ in \cite{McCPoo}) then yields that there are exactly four rational points; only one of these corresponds to a non-degenerate 
(i.e.\ 4 branch points) cover.\\ \\
{\bf Remark:} The polynomial in Theorem \ref{2234} will be used in a forthcoming paper with D.\ Neftin (\cite{KN}) to show that $M_{11}$ is $\qq$-admissible.
%
\section{Number fields with Galois group $M_{11}$ and small discriminant}
Of course, the polynomials computed above lead to infinitely many polynomials with Galois group $M_{11}$ over $\qq$,
via Hilbert's irreducibility theorem. Among those, polynomials that lead to number fields with small discriminant 
(either with regard to absolute value, or with regard to the number or absolute value of the prime divisors) are traditionally of 
special interest. Systematic collections for number fields with small discriminant and Galois groups of small degree can be found in the databases by 
Jones and Roberts (\cite{JR}), and by Kl\"uners and Malle (\cite{KlM}).
\\ \\
The very small coefficients of the polynomial $f(x,t)$ from Theorem \ref{m11small} lead to nice number fields upon specializing $t\mapsto t_0\in \qq$ appropriately. 
In the following lemma, we collect a few sample polynomials which lead to nice field discriminants:

\begin{lemma}
With $f$ as in Theorem \ref{m11small}, let $f_0(x):= f(x,1/8)$, $f_1(x):=f(x, 25/4)$, $f_2(x):=f(x,25/2)$ and $f_3(x):=f(x,2\cdot 3^6/5^3) \in \qq[X]$.\\
Then $Gal(f_i\mid \qq) \cong M_{11}$ (for $i=0,...,3$).
Furthermore, if $\xi$ is a root of $f_0$, then $\qq(\xi)$ has discriminant $\Delta=2^8\cdot 3^{16}\cdot 97^4$. The root 
discriminant is therefore $\Delta^{1/12} = 31.55...$\\
In the same way, if $\xi$ is a root of $f_1$, $\qq(\xi)\mid\qq$ is ramified only above $p=2,3,5$ and $11$, i.e.\ the 
prime divisors of $|M_{11}|$. The root discriminant of $\qq(\xi)$ is equal to $(2^8\cdot 3^{12}\cdot 5^{10}\cdot 11^6)^{1/12}=60.39...$\\
For $\xi$ a root of $f_2$, $\qq(\xi)\mid \qq$ is ramified only over 2,3 5 and 7, with root discriminant $(2^8\cdot 3^{16}\cdot 5^{10}\cdot 7^4)^{1/12}=50.23...$\\
Finally, a root field of $f_3$ has discriminant $45513961^4$, and is therefore ramified only over the prime $45513961$.
\end{lemma}
The discriminants of $f_0$, $f_1$ and $f_2$ above (and many others, for other specializations of $t$) are significantly smaller than the smallest previously known values for $M_{11}$-extensions 
(In the database \cite{KlM}, the smallest root discriminant is $(661^8)^{1/12}=75.88...$, 
however with the remarkable property that only the prime $661$ ramifies).\\ In particular, the root discriminant of
$f_0$ is smaller than $8\pi e^\gamma = 44.76...$; it is known that, if the Generalized Riemann Hypothesis holds, only
finitely many number fields have a root discriminant smaller than this value.\\
The polynomial $f_3$ above provides the second known instance of an $M_{11}$ number field ramified only at one prime, after the aforementioned Kl\"uners-Malle example.
\\ \\
For the five point covers as well, suitable specializations of $t$ lead to number fields with small discriminant:
\begin{lemma}
With $f$ as in Theorem \ref{m11_5er}, let 
$f_0(x):=f(x,40)$, 
$f_1(x):=f(x,0)$ and $f_2(x):=f(x, -608)$. Then $f_0$ has Galois group $M_{11}$ over $\qq$, and the discriminant of a root field is $2^{22}\cdot 6451^4$, i.e.\ 
the root discriminant is $(2^{22}\cdot 6451^4)^{1/12}=66.33...$\\
Furthermore, $f_1$ and $f_2$ have Galois group $M_{10}$ (acting transitively on 12 points) over $\qq$.\\
The discriminants of the corresponding root fields are $2^{38}\cdot 3^{12}$ for $f_1$, and $2^{38}\cdot 5^{12}$ for $f_2$.\\
In particular, the root discriminant for $f_1$ is equal to $3\cdot 2^{19/6}=26.93...$\\
Finally, with $g$ as in Theorem \ref{m11_5er}, let $g_0(x):=g(x,440/27)$. Then a root field of $g_0$ has discriminant $5^4\cdot 11^{12}\cdot 37^4$, i.e.\ root discriminant $62.67...$
\end{lemma}
Finally, the following specializations lead to very small {\it Galois} root discriminants, i.e.\ root discriminants of the splitting field of the corresponding polynomial:
\begin{lemma}
\label{GRD}
With $f$ as in Theorem \ref{m11small}, the root discriminant of the splitting field of $f(x,5^4/2^5)$ is approximately $108.53$. For the splitting field of $f(x,1/8)$, it is approximately $120.90$. Both polynomials have Galois group $M_{11}$.
\end{lemma}
\section{Suggestions for further research}
If there is anything to learn from the above examples, it is that Hurwitz spaces can contain rational points even though the theoretical criteria do not suffice to show 
this. 
As weak a statement as this is, it may give some motivation to explicitly search for rational points on other Hurwitz spaces of Galois theoretic interest.
A very interesting Hurwitz space would be the one associated to the class 5-tuple $(2A,2A,2A,2A,3A)$ in $M_{23}$ (note the similarity with the above $M_{11}$ tuple!).
This is again a genus zero tuple, and the only braid orbit is of length 21456. Finding explicit algebraic equations for the reduced Hurwitz spaces might be computationally hard as
the equations could have very large degree. A complex approximation to a cover in this family, which can serve as a starting point for further research, was given in \cite{Koe} 
(numerical values are available at \texttt{http://opus.bibliothek.uni-wuerzburg.de/frontdoor/index/index/docId/10014}).
\\ \\
It would also be desirable to have a database of genus zero covers with ``interesting" Galois group, in the spirit of the above computations. Ideally, for each class tuple of genus zero
in a primitive permutation group $G$, enumerate the braid orbits for genus zero tuples (this is an ongoing project; it is complete e.g.\ for affine groups, cf.\ \cite{Wang}) and for 
as many orbits as possible, find a polynomial defined over $\qq(t)$ (corresponding to a $\qq$-point on the Hurwitz space), or at least a ``nice" defining equation for the Hurwitz space.
In particular, classify which of these (reduced) Hurwitz spaces contain infinitely many points.\\ \\

\section{Appendix: Data for the computations in Section \ref{2235}}
To enable the reader to reproduce the computational results of the previous sections as far as possible, we give the explicit algebraic dependencies that occurred in the computations
for the family in Section \ref{2235}.
\subsection{Algebraic dependency between two coefficients}
With $A_1$, $A_3$ as in Section\ref{2235}, and $B_3:=A_3-8$, we have
\begin{center}
\begin{footnotesize}
$
0= f(A_1,B_3):=(A_1^2 - 11/2 A_1 + 19/4) B_3 ^{13} + (-5155/194 A_1^3 + 38895/97 A_1^2 - 594531/388 A_1 + 117796/97) B_3 ^{12}
+ (309513/388 A_1^4 - 9154407/776 A_1^3 + 7117272/97 A_1^2 - 19074030/97 A_1 + 13563714/97) B_3 ^{11} + (-29442231/776 A_1^5 + 
185861157/388 A_1^4 - 1019441127/388 A_1^3 + 1649606787/194 A_1^2 - 1515954054/97 A_1 + 933636696/97) B_3 ^{10} + (1191406975/1552 A_1^6 - 5109923301/388 A_1^5 + 33698403819/388 A_1^4 - 301286344 A_1^3 + 62107944384/97 A_1^2 - 81450055872/97 A_1 + 42428352364/97) B_3^9 + (-17975078141/1552 A_1^7 + 159298889831/776 A_1^6 - 626715027771/388 A_1^5 + 
    685163096102/97 A_1^4 - 1780162376908/97 A_1^3 + 2894712547152/97 A_1^2 - 2942324497652/97 A_1 + 1306161250816/97) B_3^8 + (204145756777/1552 A_1^8 - 1678473797789/776 A_1^7 + 6526135159681/388 A_1^6 - 8046267839200/97 A_1^5 + 26565429617995/97 A_1^4 - 57183503610562/97 A_1^3 + 
    79063401316684/97 A_1^2 - 68272197545824/97 A_1 + 26855456923504/97) B_3^7 + (469693235489/776 A_1^9 - 782565887433/388 A_1^8 - 11682813814197/388 A_1^7 + 29166319066500/97 A_1^6 - 141250725748833/97 A_1^5 + 445482866157504/97 A_1^4 - 917628514989096/97 A_1^3 + 1204178238063816/97 A_1^2 -
    971365312264560/97 A_1 + 359198224489216/97) B_3^6 + (-1261131571049/97 A_1 ^{10} + 153779721906959/776 A_1^9 - 465742701080757/388 A_1^8 + 731401101761739/194 A_1^7 - 601339328468574/97 A_1^6 + 77957164893978/97 A_1^5 + 2479783644251400/97 A_1^4 - 7049617904044248/97 A_1^3 + 
    10157509301026896/97 A_1^2 - 8389064431939136/97 A_1 + 3095129706793984/97) B_3^5 + (-93086889753555/1552 A_1 ^{11} + 364159057667085/776 A_1 ^{10} + 54396727532199/388 A_1^9 - 1305918270868851/97 A_1^8 + 6253214552234415/97 A_1^7 - 15296610106462644/97 A_1^6 + 21936576360454164/97 A_1^5 - 
    13052999652175512/97 A_1^4 - 16128803133158976/97 A_1^3 + 43728774887217888/97 A_1^2 - 42958925055571968/97 A_1 + 16980079764013056/97) B_3^4 + (-70734693371649/1552 A_1 ^{12} - 579921595821849/776 A_1 ^{11} + 4348257884577825/388 A_1 ^{10} - 8938247753354073/194 A_1^9 + 
    4108909424631336/97 A_1^8 + 24282392224706388/97 A_1^7 - 102830373528281748/97 A_1^6 + 201711501988308216/97 A_1^5 - 227430471125360448/97 A_1^4 + 114947452249545792/97 A_1^3 + 50586821245936128/97 A_1^2 - 118329393903808512/97 A_1 + 57202460850388992/97) B_3^3 + 
    (-21092994179621/1552 A_1 ^{13} - 467534302099537/776 A_1 ^{12} - 201800064648639/97 A_1 ^{11} + 7680831213174454/97 A_1 ^{10} - 44955084124781402/97 A_1^9 + 114433297785064404/97 A_1^8 - 105663894944016504/97 A_1^7 - 165159884235814752/97 A_1^6 + 679903017365007504/97 A_1^5 - 
    1050939682750989344/97 A_1^4 + 875155021225057024/97 A_1^3 - 307070842274199552/97 A_1^2 - 108929402405322752/97 A_1 + 107749883112325120/97) B_3^2 + (-169500204830/97 A_1 ^{14} - 51905365308787/388 A_1 ^{13} - 490968932654785/194 A_1 ^{12} + 479499641453872/97 A_1 ^{11} + 
    21966452336596532/97 A_1 ^{10} - 166218554568409676/97 A_1^9 + 541320630983061936/97 A_1^8 - 926539021804859664/97 A_1^7 + 674622852518559840/97 A_1^6 + 528310839806064064/97 A_1^5 - 1875720529374405376/97 A_1^4 + 2126826674831042560/97 A_1^3 - 1204200478741200896/97 A_1^2 + 
    195407155555205120/97 A_1 + 86785744961536000/97) B_3 - 121526618175/1552 A_1 ^{15} - 7092226827585/776 A_1 ^{14} - 62812640495007/194 A_1 ^{13} - 319321013801742/97 A_1 ^{12} + 2010263178874275/97 A_1 ^{11} + 21321807694457418/97 A_1 ^{10} - 215184880237281660/97 A_1^9 + 815640836930174088/97 A_1^8 - 
    1691389128726897840/97 A_1^7 + 1993736165732507232/97 A_1^6 - 972260208590266368/97 A_1^5 - 762544685275379712/97 A_1^4 + 1684771171812900864/97 A_1^3 - 1264525687424286720/97 A_1^2 + 390535852326912000/97 A_1
 .$
 \end{footnotesize}
\end{center}
 \subsection{Weierstrass normal form of the Hurwitz curve}
Riemann-Roch space computations yield a rational field $\qq(y)$ of index 2 in $\qq(A_1,B_3)$; more precisely, $y$ fulfills the equation
\begin{center}
\begin{footnotesize}
$0=g( y,A_1):=( y^{13} - 7244/2997  y ^{12} + 562217/221778  y ^{11} - 51348491/32823144  y ^{10} + 1556839345/2428912656  y^9 - 33211333907/179739536544  y^8 + 42799573753/1108393808688  y^7 - 1457964326767/246063425528736  y^6 
+ 24313304310661/36417386978252928  y^5 - 16315958898565/299431848487857408  y^4 + 414239909321917/132947740728608689152  y^3 - 431545438089731/3689299805218891123968  y^2 + 5498264158963009/2184065484689583545389056  y
- 1115096877547877/53873615289009727452930048)  A_1^2 + (-4/3  y ^{13} + 10586/2997  y ^{12} - 
    432298/110889  y ^{11} + 41424745/16411572  y ^{10} - 658795045/607228164  y^9 + 29408140453/89869768272  y^8 - 26267657239/369464602896  y^7 + 5532755656967/492126851057472  y^6 - 11753348201935/9104346744563232  y^5 + 47566719303965/449147772731786112  y^4 - 
    198721227411479/33236935182152172288  y^3 + 3180510221756591/14757199220875564495872  y^2 - 1150005123206249/273008185586197943173632  y
    + 10764202407935/420887619445388495726016)  A_1 + 4/9  y ^{13} - 3860/2997  y ^{12} + 168310/110889  y ^{11} - 4195825/4102893  y ^{10} + 
    69248305/151807041  y^9 - 3220640269/22467442068  y^8 + 35967907025/1108393808688  y^7 - 81843047635/15378964095546  y^6 + 11434921876475/18208693489126464  y^5 - 11716420884695/224573886365893056  y^4 + 32268224062031/11078978394050724096  y^3 - 
    91273117744715/922324951304722780992  y^2 + 53821012039675/34126023198274742896704  y$
 \end{footnotesize}
 \end{center}

 For the further computations, one needs to express $y$ as a rational function in $A_1$ and $B_3$. Even though this function will be of large degree, computer programs like Magma find
 the unique root $g(t,A_1)$ in $\mathbb{F}_p(A_1,A_3)$ quickly for small primes $p$; after this, use the Chinese remainder theorem.\\
 With the new and easier parametrization of the Hurwitz curve, it is not difficult anymore to find a rational field $\qq(w)$ of index 3 as well.
With
\begin{center}
 \begin{footnotesize} $w:=(36963/5  y ^{13} - 268028/15  y ^{12} + 562217/30  y ^{11} - 51348491/4440  y ^{10} + 311367869/65712  y^9 - 33211333907/24313440  y^8 + 42799573753/149932880  y^7 - 1457964326767/33285099360  y^6 + 24313304310661/4926194705280  y^5 - 9789575339139/24302560546048  y^4 + 
        414239909321917/17983894804075520  y^3 - 431545438089731/499053080813095680  y^2 + 5498264158963009/295439423841352642560  y - 1115096877547877/7287505788086698516480)/( y ^{11} - 969/370  y ^{10} + 108807/54760  y^9 - 179903/253265  y^8 + 34830867/299865760  y^7 + 
        46656211/11095033120  y^6 - 10076229327/1642064901760  y^5 + 85093388909/60756401365120  y^4 - 788499220351/4495973701018880  y^3 + 4418000938479/332702053875397120  y^2 - 28673502113207/49239903973558773760  y + 2152840481587/182187644702167462912) x + (-24642/5  y ^{13} 
        + 195841/15  y ^{12} - 216149/15  y ^{11} + 8284949/888  y ^{10} - 131759009/32856  y^9 + 29408140453/24313440  y^8 - 78802971717/299865760  y^7 + 5532755656967/133140397440  y^6 - 2350669640387/492619470528  y^5 + 9513343860793/24302560546048  y^4 - 
        198721227411479/8991947402037760  y^3 + 3180510221756591/3992424646504765440  y^2 - 1150005123206249/73859855960338160640  y + 2152840481587/22773455587770932864)/( y ^{11} - 969/370  y ^{10} + 108807/54760  y^9 - 179903/253265  y^8 + 34830867/299865760  y^7 + 
        46656211/11095033120  y^6 - 10076229327/1642064901760  y^5 + 85093388909/60756401365120  y^4 - 788499220351/4495973701018880  y^3 + 4418000938479/332702053875397120  y^2 - 28673502113207/49239903973558773760  y + 2152840481587/182187644702167462912)$, 
        \end{footnotesize}\end{center}  one then gets the parameterization
    
    $$w^2=(2\cdot 37)^3(y+2/37)(y^2-43/296y+4/1369)$$
   \\ 
   From here, obvious linear transformations yield the defining equation
$ W^2=Y^3-675Y-1250$ given in Section \ref{2235}. 

\ \\
{\bf Acknowledgement}: I would like to thank David Roberts for valuable remarks on root discriminants, including the values for Galois root discriminants in Lemma \ref{GRD}.

\end{document}